\documentclass[a4paper]{article}

\usepackage[cyr]{aeguill}
\usepackage[english]{babel}
\usepackage[utf8]{inputenc}
\usepackage[T1]{fontenc}

\usepackage[a4paper,top=3cm,bottom=2cm,left=3cm,right=3cm,marginparwidth=1.75cm]{geometry}

\usepackage{amsmath}
\usepackage{graphicx}
\usepackage[colorinlistoftodos]{todonotes}
\usepackage[colorlinks=true, allcolors=blue]{hyperref}
\usepackage{amssymb}
\usepackage{amsthm}
\usepackage{bbold}
\usepackage[all]{xy}
\usepackage{dirtytalk}
\usepackage[titletoc,title]{appendix}
\usepackage{ stmaryrd }

\usepackage{comment}

\newtheorem{theorem}{Theorem}[section]

\theoremstyle{definition}
\newtheorem{definition}{Definition}[section]

\title{\textbf{A note on the RAGE Theorem and phase-averaged dispersion for the Fibonacci Hamiltonian}}
\author{Gaétan Leclerc}

\begin{document}

\maketitle

\begin{abstract}
We find a weaker condition on spectral measures, \say{eventual absolute continuity}, that ensure quantum delocalization as in the RAGE Theorem in the case of purely absolutely continuous spectrum. We then adapt these idea to strongly improve some phase-averaged delocalization bounds for the Fibonacci quasicrystal. 
\end{abstract}

\section{On the RAGE Theorem}

In these notes, we discuss the spectral Theory of bounded selfadjoints operators, and in particular we discuss the Fibonacci Hamiltonian. A complete source on this topic is given by the books \cite{DF22,DF24}. For an introduction, on can consult the survey \cite{Da17}. \\

Let $\mathcal{H}$ be a Hilbert space. Let $H : \mathcal{H} \rightarrow \mathcal{H}$ be a Hamiltonian: to simplify, suppose that $H$ is bounded and selfadjoint. For $\psi \in \mathcal{H}$, we denote by $\mu_\psi$ the spectral measure associated with $H$ and $\psi$. This finite measure is supported on the spectrum $\sigma(H) \subset \mathbb{R}$ of $H$ (we will denote $\mu_\psi \in \mathcal{P}(\sigma(H))$). We write $\psi \in \mathcal{H}_{ac}$ if $\mu_\psi $ is absolutely continuous with respect to the Lebesgue measure ($\mu_\psi \ll \lambda$). The set $\mathcal{H}_{ac}$ is a closed linear subspace of $\mathcal{H}$. \\

The famous RAGE Theorem, in the (easy) absolutely continuous case, state that:

\begin{theorem}
Let $\psi \in \mathcal{H}_{ac}$. Then for every $\varphi \in \mathcal{H}$, $$ \langle e^{it H} \psi , \varphi \rangle \underset{t \rightarrow \infty}{\longrightarrow} 0.$$
\end{theorem}

\begin{proof}
Let $\psi \in \mathcal{H}_{ac}$. Notice that for any Borel set $E$ with zero lebesgue measure $|E|=0$, we have vanishing of the spectral projector:
$$ \|\mathbb{1}_E(H) \psi \|_{\mathcal{H}}^2 = \langle \mathbb{1}_E^2(H) \psi,\psi\rangle_{\mathcal{H}} = \int_{\sigma(H)} \mathbb{1}_E(t) d\mu_\psi(t) = \mu_\psi(E)=0 $$
by absolute continuity of $\mu_\psi$. Now, consider the linear form: 
$$ f \in C^0(\sigma(H),\mathbb{C}) \longmapsto \langle f(H) \psi,\varphi \rangle_{\mathcal{H}} \in \mathbb{C}.$$
The linear form is continuous for the topology of $L^\infty(\sigma(H),\mathbb{C})$, since $\|f(H)\|_{\mathcal{H} \rightarrow \mathcal{H}} \leq \|f\|_{L^\infty(\sigma(H),\mathbb{C})}$ by the usual properties of the functionnal calculus. By Riesz representation theorem, there exists a complex measure $\mu_{\psi,\varphi}$ such that $ \langle f(H) \psi,\varphi \rangle_{\mathcal{H}} = \int f d\mu_{\varphi,\psi}$. But notice that for any borel set $E$ with $|E|=0$, we have
$$ \mu_{\varphi,\psi}(E) = \langle \mathbb{1}_E(H) \psi,\varphi \rangle_{\mathcal{H}} = 0. $$
It follows that $\mu_{\varphi,\psi}$ is absolutely continuous with respect to the lebesgue measure. We then find
$$ \langle e^{it H} \psi,\varphi \rangle = \widehat{\mu}_{\psi,\varphi}(t) \rightarrow 0, $$
by the Riemann-Lebesgue Lemma.
\end{proof}

The condition that $\mu_\psi$ is absolutely continuous can actually be relaxed. For any $N \geq 1$, we define another Hamiltonian $H_k = I \otimes \dots \otimes H \otimes \dots \otimes I$ acting on $\mathcal{H}^{\otimes N}$, and we then define $\widehat{H} := \sum_k H_k$. One can then look at the spectral measure of $\psi^{\otimes N}$ associated to $\widehat{H}$, denoted $\mu_{\psi^{\otimes N}}$. Recall that we have the formula: $$ \mu_{\psi^{\otimes N}} = \mu_{\psi}^{* N} ,$$ where $*$ denotes a convolution. In fact, in general, we have $ \mu_{\psi \otimes \varphi} = \mu_\psi * \mu_\psi $. This is because:
$$ \widehat{\mu_{\psi \otimes \varphi}}(t) = \langle e^{i t (I \otimes H + H \otimes I)} \psi \otimes \varphi, \psi \otimes \varphi\rangle_{\mathcal{H} \otimes \mathcal{H}} =\langle  (e^{it H}\psi) \otimes (e^{it H}\varphi), \psi \otimes \varphi\rangle_{\mathcal{H} \otimes \mathcal{H}} $$ $$ = \langle e^{it H}\psi , \psi \rangle_\mathcal{H} \cdot \langle  e^{it H}\varphi , \varphi \rangle_{\mathcal{H}} = \widehat{\mu_\psi}(t) \widehat{\mu_\varphi}(t) = \widehat{\mu_\psi * \mu_\varphi }(t) .$$

\begin{definition}
Let us define the eventually absolutely continuous subspace as:
$$ \mathcal{H}_{eac} := \overline{ \{ \psi \in \mathcal{H} \ | \ \exists N \geq 1, \ \mu_{\psi^{\otimes N}} \ll \lambda \} }. $$
This is a closed linear subspace of $\mathcal{H}$.
\end{definition}

\begin{theorem}[RAGE]
Suppose that $\psi \in \mathcal{H}_{eac}$. Then for every $\varphi \in \mathcal{H}$, $$ \langle e^{it H} \psi , \varphi \rangle \underset{t \rightarrow \infty}{\longrightarrow} 0.$$
\end{theorem}

\begin{proof}
Let $\psi \in \mathcal{H}$ be such that $\mu_{\psi^{\otimes N}} = \mu_\psi^{* N}$ is absolutely continuous for some large enough $N$. Let $\varphi \in \mathcal{H}$. Then:
$$ \langle e^{it H} \psi , \varphi \rangle_{\mathcal{H}}^N = \langle (e^{it H} \psi)^{\otimes N} , \varphi^{\otimes N} \rangle_{\mathcal{H}^{\otimes N}} = \langle e^{it \widehat{H}} \psi^{\otimes N} , \varphi^{\otimes N} \rangle_{\mathcal{H}^{\otimes N}} \underset{t \rightarrow \infty}{\longrightarrow} 0$$
by the RAGE Theorem, since $\mu_{\psi^{\otimes N}}$ is absolutely continuous. We conclude for any $\psi \in \mathcal{H}_{aec}$ by density.
\end{proof}

\underline{\textbf{Question}}:
What is the space $\mathcal{H}_{eac}^\perp \subset \mathcal{H}$ ?

\section{On the Fibonacci Hamiltonian}

Adapting the idea of the previous proof, we can establish strong phase-averaged delocalization for the Fibonacci Hamiltonian.
Consider the Fibonacci Hamiltonian $H_\omega : \ell^2(\mathbb{Z},\mathbb{C}) \rightarrow \ell^2(\mathbb{Z},\mathbb{C})$, with phase $\omega \in \mathbb{T} = \mathbb{R}/\mathbb{Z}$, for some small enough fixed coupling constant $V>0$. Recall that it is given by the formula:

$$ H_\omega \psi(n) = \psi(n+1)+\psi(n-1)+V \mathbb{1}_{[1-\alpha,1)}(\omega+n \alpha \text{ mod } 1) \psi(n) $$
where $\alpha := (\sqrt{5}-1)/2$ is the inverse of the golden ratio.
These are (uniformly) bounded and selfadjoint linear operators. Notice that, denoting $S \psi(n) = \psi(n+1)$ the shift (unitary in $\ell^2(\mathbb{Z})$), we have $S^n H_\omega (S^*)^n = H_{\omega+\alpha n}$. In general, all the operators $H_\omega$ have the same spectrum $\Sigma_V \subset \mathbb{R}$, independantly on the choice of the phase. Define $\mu_\omega \in \mathcal{P}(\Sigma_V)$ to be the spectral measure associated with $\delta_0 \in \ell^2(\mathbb{Z},\mathbb{C})$ and $H_\omega$. Define also the density of state measure $\mu \in \mathcal{P}(\Sigma_V)$ by the formula $$\mu = \int_{\omega \in \mathbb{T}} \mu_\omega d\omega.$$ It was proved in \cite{Le25} that
$$ |\widehat{\mu}(\xi)| \leq |\xi|^{-\varepsilon} ,$$
for some $\varepsilon>0$ and for large enough $\xi>0$. In other words, we have
$$ \Big| \int_{\omega \in \mathbb{T}} \langle e^{it H_\omega} \delta_0,\delta_0\rangle_{\ell^2(\mathbb{Z},\mathbb{C})} d\omega \Big| \leq |t|^{-\varepsilon}$$
for large enough $t$. It follows that $\mu^{*N} $ is absolutely continuous for $N$ large enough, since its Fourier transform will be e.g. in $L^2$. Adapting the previous proof, we will show the following consequence:
\begin{theorem}[Escape of mass]
For any $\ell^1$-localized state $\psi \in \ell^1(\mathbb{Z},\mathbb{C})$, and for any $\varphi \in \ell^2(\mathbb{Z},\mathbb{C})$, we have:
$$ \int_{\omega \in \mathbb{T}} \langle e^{itH_\omega} \psi, \varphi \rangle_{\ell^2(\mathbb{Z})} d\omega \underset{t \rightarrow \infty}{\longrightarrow} 0 .$$
\end{theorem}

\begin{proof}
Notice that, since $$\mu = \int_\omega \mu_\omega d\omega,$$ we have $$\mu^{*N} = \int_{\vec{\omega} \in \mathbb{T}^N} \mu_{\vec{\omega}} d\vec{\omega} ,$$
where $\mu_{\vec{\omega}} := \mu_{\omega_1} * \dots * \mu_{\omega_N}$. Define the naturally associated Hamiltonian $H_{\vec{\omega}} : \ell^2(\mathbb{Z}^N,\mathbb{C}) \longrightarrow \ell^2(\mathbb{Z}^N,\mathbb{C}) $ to be
$$ H_{\vec{\omega}} = (H_{\omega_1} \otimes I \otimes \dots \otimes I) + \dots + (I \otimes \dots \otimes I \otimes H_{\omega_N}) .$$
Notice then that, for any Borel set $E\subset \mathbb{R}$ with zero lenght $|E|=0$, we have:
$$ \int_{\vec{\omega} \in \mathbb{T}^N} \| \mathbb{1}_{E}(H_{\vec{\omega}}) \delta_0^{\otimes N} \|_{\ell^2(\mathbb{Z}^N)} d\vec{\omega} = 0.  $$
Indeed, by Cauchy-Schwarz, by definition of the spectral measures, and by absolute continuity of $\mu^{*N}$:
$$  \Big(\int_{\vec{\omega} \in \mathbb{T}^N} \| \mathbb{1}_{E}(H_{\vec{\omega}}) \delta_0^{\otimes N} \|_{\ell^2(\mathbb{Z}^N)} d\vec{\omega}\Big)^2 \leq \int_{\vec{\omega} \in \mathbb{T}^N} \| \mathbb{1}_{E}(H_{\vec{\omega}}) \delta_0^{\otimes N} \|_{\ell^2(\mathbb{Z}^N)}^2 d\vec{\omega} $$
$$ = \int_{\vec{\omega} \in \mathbb{T}^N} \langle \mathbb{1}_{E}^2(H_{\vec{\omega}}) \delta_0^{\otimes N} ,\delta_0^{\otimes N} \rangle_{\ell^2(\mathbb{Z}^N)} d\vec{\omega} $$
$$ = \int_{\vec{\omega} \in \mathbb{T}^N} \int_{\Sigma_V} \mathbb{1}_E(t) d\mu_{\vec{\omega}}(t) d\vec{\omega} = \mu^{*N}(E)=0. $$
Be carefull that this doesn't imply that any of the measures $\mu_\omega$ are absolutely continuous, and that they could very well all have atoms.
Let us now fix any $\phi \in \ell^2(\mathbb{Z}^N,\mathbb{C})$, and consider the linear form:

$$ f \in C^0(\Sigma_V,\mathbb{C}) \longmapsto \int_{\vec{\omega} \in \mathbb{T}^N} \langle f(H_{\vec{\omega}}) \delta_0^{\otimes N} ,\phi \rangle_{\ell^2(\mathbb{Z}^N)} d\vec{\omega} \in \mathbb{C}.$$

This linear form is continuous: by Riesz representation theorem, there exists a complex measure $\mu_\phi$ supported on $\Sigma_V$ such that 
$$ \int_{\vec{\omega} \in \mathbb{T}^N} \langle f(H_{\vec{\omega}}) \delta_0^{\otimes N} ,\phi \rangle_{\ell^2(\mathbb{Z}^N)} d\vec{\omega} = \int_{\Sigma_V} f d\mu_\phi. $$
But notice that, for any Borel set $E\subset \mathbb{R}$ of zero lenght $|E|=0$, we have
$$ |\mu_\varphi(E)| = \Big| \int_{\vec{\omega} \in \mathbb{T}^N} \langle \mathbb{1}_E(H_{\vec{\omega}}) \delta_0^{\otimes N},\phi \rangle_{\ell^2(\mathbb{Z}^N)} d\vec{\omega} \Big| \leq \int_{\vec{\omega} \in \mathbb{T}^N} \| \mathbb{1}_E(H_{\vec{\omega}}) \delta_0^{\otimes N} \|_{\ell^2(\mathbb{Z}^N)} d\vec{\omega} \cdot \|\phi \|_{\ell^2(\mathbb{Z}^N)}  = 0. $$
By the Radon-Nikodym theorem for complex measures, it follows that $\mu_\phi$ is absolutely continuous. In particular, we have $ \widehat{\mu_\phi}(t) \rightarrow 0 $, which means that, for any $\phi \in \ell^2(\mathbb{Z}^d)$, we have:
$$ \int_{\vec{\omega}} \langle e^{it H_{\vec{\omega}}} \delta_0^{\otimes N}, \phi \rangle_{\ell^2(\mathbb{Z}^N)} d\vec{\omega} \underset{t \rightarrow \infty}{\longrightarrow} 0 . $$
In particular, if $\phi$ is of the form $\phi=\varphi^{\otimes N}$, we find
$$ \Big(\int_{\omega \in \mathbb{T}} \langle e^{it H_{\omega}} \delta_0, \varphi \rangle_{\ell^2(\mathbb{Z})} d{\omega} \Big)^N = \int_{\vec{\omega} \in \mathbb{T}^N} \langle e^{it H_{\vec{\omega}}} \delta_0^{\otimes N}, \varphi^{\otimes N} \rangle_{\ell^2(\mathbb{Z}^N)} d\vec{\omega} \underset{t \rightarrow \infty}{\longrightarrow} 0. $$
Let us now conclude the proof. Take any $\psi \in \ell^1(\mathbb{Z}) \subset \ell^2(\mathbb{Z})$, and any $\varphi \in \ell^2(\mathbb{Z})$. We can write $\psi$ as a sum (absolutely converging for $\|\cdot\|_{\ell^2(\mathbb{Z})}$):
$$ \psi = \sum_{k \in \mathbb{Z}} \psi(k) S^{-k} (\delta_0), $$
where $S$ is the shift. Hence, doing a change of variable in the integrals, using the relation $S H_\omega S^{-1} = H_{\omega+\alpha}$, and noticing that $\sum_{k \in \mathbb{Z}} \overline{\psi(k)} S^k \varphi \in \ell^2(\mathbb{Z})$:
$$ \int_{\omega \in \mathbb{T}} \langle e^{it H_\omega} \psi,\varphi \rangle d\omega = \sum_{k \in \mathbb{Z}} \psi(k) \int_{\omega \in \mathbb{T}} \langle e^{it H_\omega} S^{-k} \delta_0, \varphi \rangle_{\ell^2(\mathbb{Z})} d\omega  $$
$$ =  \sum_{k \in \mathbb{Z}} \psi(k) \int_{\omega \in \mathbb{T}} \langle S^{-k} e^{it H_{\omega+k\alpha}}  \delta_0, \varphi \rangle_{\ell^2(\mathbb{Z})} d\omega $$
$$ =  \sum_{k \in \mathbb{Z}} \psi(k) \int_{\omega \in \mathbb{T}} \langle  e^{it H_{\omega}}  \delta_0, S^k \varphi \rangle_{\ell^2(\mathbb{Z})} d\omega  $$ $$ = \int_{\omega \in \mathbb{T}} \Big\langle e^{itH_\omega} \delta_0, \sum_{k \in \mathbb{Z}} \overline{\psi(k)} S^k \varphi \Big\rangle_{\ell^2(\mathbb{Z})} d\omega \underset{t \rightarrow \infty}{\longrightarrow 0} .$$
\end{proof}

\section*{Acknowledgments}

We thank Tuomas Sahlsten, Frederic Naud, Mostafa Sabri and Jake Fillman for discussions during the redaction of this note. It was mentionned to me that a \say{dual} work, in the sense that we restrict the absolutely continuous spectrum instead of extending it, can be found in the work of Avron and Simon \cite{AS81}.

\end{document}